\documentclass[12pt]{article}

\usepackage{amssymb}
\usepackage{latexsym}
\usepackage{graphicx}
\usepackage{amsmath}
\usepackage{amsthm}
\usepackage{amsfonts}
\usepackage{amscd}

\theoremstyle{plain}
\newtheorem{theorem}{Theorem}[section]

\newtheorem{lemma}{Lemma}[section]
\newtheorem{definition}{Definition}[section]
\newtheorem{corollary}{Corollary}[section]

\title{\ 
A relation between the shape of a permutation and the shape of the base poset derived from the Lehmer codes}
\author{Masaya Tomie}
\date{University of Morioka, Takizawa-mura, Iwate 020-0183, Japan 
(e-mail: tomie@morioka-u.ac.jp)}

\begin{document}

\maketitle

\begin{abstract}

For a permutation $\omega \in S_{n}$ Denoncourt constructed a poset $M_{\omega}$ which is the set of join-irreducibles of  Lehmer codes of the permutations in 
$[e, \omega]$ in the inversion order on $S_{n}$. 
In this paper we show that 
$M_{\omega}$ is a $B_{2}$-free poset if and only if $\omega$ is a  $3412-3421$-avoiding permutation.

\end{abstract}

\section{Introduction}

Let $P, Q$ be a poset. A subposet $R \subset P$ is called {\it a $Q$-pattern subposet} if $R \simeq Q$ as a poset. We say that {\it $P$ is $Q$-free} if $P$ has no $Q$-pattern subposets. 
The number of $1 + 3$-free and $2 + 2$-free posets with $n$ elements is $\frac{1}{n+1} \binom{2n}{n}$ the $n$-th Catalan number \cite{stanley}. 

For $\sigma \in S_{n}, \pi \in S_{k}$ with $k < n$, we say a permutation $\sigma$ is a 
$\pi$ avoiding permutation if $st(\sigma(i_{1}) \sigma(i_{2}) \cdots \sigma(i_{k})) \neq \pi(1) \pi(2) \cdots \pi(k)$ 
for any  $1 \le i_{1} < i_{2} , \cdots < i_{k} \le n$.  
The number of $\pi$ avoiding permutation in $S_{n}$ is $\frac{1}{n+1} \binom{2n}{n}$ for all $\pi \in S_{3}$ \cite{knuth-1} \cite{knuth-2}.
In this paper we consider the relation  {\it $B_{2}$-free posets} where $B_{2}$ is Boolean algebra of rank $2$ and $3412-3421$-avoiding permutations by considering Lehmer codes. 

In \cite{denoncourt} Denoncourt showed that the set of Lehmer codes for permutations in $\Lambda_{\omega}$ ordered by the product order on $\mathbb{N}^{n}$ is a distributive lattice where $\Lambda_{\omega} = \{ \sigma | {\rm Inv}(\sigma) \subset {\rm Inv}(\omega) \}$ and he also gave the  expression of  $M_{\omega}$ which is  the set of  join-irreducibles of the set of Lehmer codes for  $\Lambda_{\omega}$. 
 
In this paper we focus on the relation between the shape of $M_{\omega}$ and that of $\omega$ and obtain the following result.

\begin{theorem}

$M_{\omega}$ is a $B_{2}$-free poset if and only if $\omega$ is a  $3412-3421$-avoiding permutation. 

\end{theorem}

\section{Notations and Remarks}

In this paper we use 1-line notation, this is $\omega = \omega(1) \omega(2) \cdots \omega(n)$ for $\omega \in S_{n}$. 
Put $\Lambda_{\omega} := \{ \sigma | {\rm Inv}(\sigma) \subset {\rm Inv}(\omega) \} 
$ where ${\rm Inv}(\omega) := \{ (i,j) | 1 \le i < j \le n, \omega(i) > \omega(j) \}$. In other words $\Lambda_{\omega}$ is the interval $[ e, \omega]$ in the left Bruhat order. 

We put $c_{i}(\omega) := \sharp \{j | 1 \le i < j \le n, \omega(i) > \omega(j) \}$ the number of inversions of $\omega$ with first coordinate is $i$ 
 and $c_{ij}(\omega) := \sharp \{ k | 1 \le i < k < j \le n, \omega(i) > \omega(k) \}$
 the number of inversions of $\omega$ with first coordinate is $i$ and second coordinate is between $i$ and $j$.  
The finite sequence 

\begin{center}
${\bf c}(\omega) := (c_{1}(\omega), c_{2}(\omega), \ldots , c_{n}(\omega) )$
\end{center}

 is called the {\it Lehmer code} for $\omega$ and let ${\bf c}(\Lambda_{\omega})$ be the set of Lehmer codes of permutations in $\Lambda_{\omega}$.  
In \cite{denoncourt} Denoncourt showed the following result. 

\begin{theorem}[Denoncourt]

For $\omega \in S_{n}$ the subposet ${\bf c}(\Lambda_{\omega})$ of $\mathbb{N}$ is a distributive lattice.

\end{theorem}

Let $L$ be a finite distributive lattice and $P$ the subposet of join irreducible elements in $L$. 
Then {\it the fundamental theorem for finite distributive lattices      }  states that $L \simeq J(P)$ where $J(P)$ is the poset of order ideals of $P$ ordered by inclusion \cite{birkhoff}.  In \cite{denoncourt} Denoncourt determined the set of 
join irreducible elements in ${\bf c}(\Lambda_{\omega})$.

\begin{definition}[Denoncourt]

For $i \in [n]$ such that  $c_{i}(\omega) > 0$ and for each 
$x \in [c_{i}(\omega)]$, define $m_{i,x}(\omega) \in \mathbb{N}$ coordinate-wise by 

\begin{enumerate}

\item       $\pi_{j}(m_{i,x}(\omega)) = 0$ if $(i,j) \in {\rm Inv}(\omega)$, 

\item     $\pi_{j}(m_{i,x}(\omega)) = 0$ if $j < i$, 

\item    $\pi_{j}(m_{i,x}(\omega)) = x$ if $j = i$,

\item    $\pi_{j}(m_{i,x}(\omega)) =   
                 {\rm  max}  \{ 0, x - c_{i,j}(\omega)    \}$  if  $j > i$
 and $(i,j) \notin {\rm Inv}(\omega)$

\end{enumerate}

where $\pi_{j}(m_{i,x}(\omega))$ denotes the $j$-th coordinate of $m_{i,x}(\omega)$.  Put $M_{\omega} = \{  m_{i,x}(\omega)  | 1 \le i \le n, c_{i}(\omega) > 0, x \in [c_{i}(\omega)]    $. 

\end{definition}

Let $M_{\omega} = \{ m_{i,x}(\omega) | 1 \le i \le n \ {\rm such \ that \ } c_{i}(\omega) > 0 \ {\rm and} \ x \in [c_{i}(\omega)]         \} $ and 
$C_{i}(\omega) = \{ m_{i,x}(\omega) | x \in [c_{i}(\omega)] \} $
 for $1 \le i \le n$ such that $c_{i}(\omega)  \neq 0$. 
Then $M_{\omega}$ is a subposet of $\mathbb{N}^{n}$ in the product order. 
Denoncourt showed the following results, see Corollary 5.6 and Theorem 6.6 of 
his paper \cite{denoncourt}.

\begin{theorem}[Denoncourt]

The set $M_{\omega}$ is the set of join irreducible elements of ${\bf c}(\Lambda_{\omega})$. 

\end{theorem}

\begin{lemma}[Denoucourt]\label{denoncourt-231modoki}

\begin{enumerate}

\item For $\omega \in S_{n}$ and $1 \le i < j \le n $ with $(i,j) \in {\rm Inv}(\omega)$,
 every element of $C_{i}(\omega)$ is incomparable with every element of $C_{j}(\omega)$,

\item For $\omega \in S_{n}$ and $1 \le i < j \le n $ with $(i,j) \notin {\rm Inv}(\omega)$, we have 
$m_{i,x}(\omega) > m_{j, y}(\omega)$ if and only if $y \le x - c_{i,j}(\omega)$. 
\end{enumerate}

In other words there exists a pair of comparable elements $m_{i,x}(\omega)>  m_{j, y}(\omega)$ with $1 \le i < j \le n $ if and only if  $st( \omega(i) \omega(j)  \omega(l)) = 231$ for some $j < l$. Hence we have the following corollary. 

\end{lemma}

\begin{corollary}

If $\omega$ is a $231$-avoiding permutation then $M_{\omega}$ is disjoint union of the chains. 

\end{corollary}

\section{Main Result}

In this section we give a proof of the following result. 

\begin{theorem}

$M_{\omega}$ is a $B_{2}$-free poset if and only if $\omega$ is a  $3412-3421$-avoiding permutation. 

\end{theorem}

\begin{definition}

Let $P$ be a poset. A subposet $\{ a,b,c,d \} \subset P$ with distinct elements   is called $B_{2}$-pattern subposet  if $\{ a, b, c, d \} \simeq B_{2}$ where $B_{2}$ is a Boolean algebra of rank $2$. 
We say that $P$ has a $B_{2}$-pattern if $P$ has a $B_{2}$-pattern subposet. 

\end{definition}

We will define the following poset patterns. 

\begin{definition}

\begin{enumerate}

\item 

For $1 \le i < j \le n, b < a \in [c_{i}(\omega)]$ and 
$ c < d \in [c_{j}(\omega)] $ with $a + c = b + d$ the poset 
$\{ m_{i,a}(\omega), m_{i,b}(\omega), m_{j,c}(\omega), m_{j,d}(\omega) \}$
 is called $parallelogram-pattern \ poset$ if \\ $m_{i,a}(\omega) > m_{j,d}(\omega),  m_{i,b}(\omega) > m_{j,c}(\omega)$ and the two elements $m_{i,b}(\omega)$ and $m_{j,d}(\omega)$ are incomparable. 

We say that  $M_{\omega}$ has a $parallelogram-pattern$ if $M_{\omega}$ contains a $parallelogram-pattern \ poset$. 

\item 

For $1 \le i < j \le n, b < a \in [c_{i}(\omega)]$ and 
$ c < d \in [c_{j}(\omega)] $ with $a + c = b + d$ the poset 
$\{ m_{i,a}(\omega), m_{i,b}(\omega), m_{j,c}(\omega), m_{j,d}(\omega) \}$
 is called $C_{4}-parallelogram-pattern \ poset$ if \\ $m_{i,a}(\omega) > m_{j,d}(\omega),  m_{i,b}(\omega) > m_{j,c}(\omega)$ and the two elements  $m_{i,b}(\omega) $ and $ m_{j,d}(\omega)$ are comparable. 
If $m_{i,b}(\omega)$ and $m_{j,d}(\omega)$ are comparable then we have $m_{i,b}(\omega) > m_{j,d}(\omega)$ because the $i$-th entry of $m_{i,b}(\omega)$ is $b$ and that of $m_{j,d}(\omega)$ equals to  $0$.

We say that  $M_{\omega}$ has a $C_{4}-parallelogram-pattern$ if $M_{\omega}$ contains a $C_{4}-parallelogram-pattern \ poset$. 
Especially we have $m_{i,a}(\omega) > m_{i,b}(\omega) > m_{j,d}(\omega) > m_{j,c}(\omega)$. 

\end{enumerate}

Figure \ref{fig:parallelogram-pattern-and-C4-parallelogram-one} shows the shape of the parallelogram-pattern and the 
$C_{4}$-parallelogram pattern. 

\end{definition}

A parallelogram-pattern subposet is also $B_{2}$-pattern subposet, but a 
$B_{2}$-pattern subposet is not always a parallelogram-pattern subposet. 

\begin{figure}[htbp]
\begin{center}
 \includegraphics[width=70mm]
{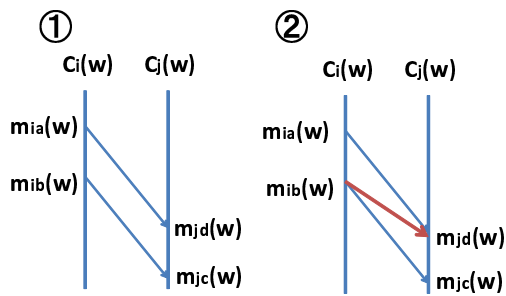}
 \caption{ Part.1 is the parallelogram-pattern and Part.2 is the $C_{4}$-parallelogram-pattern. }
 \label{fig:parallelogram-pattern-and-C4-parallelogram-one}
\end{center}
\end{figure}

The following statement is useful 
but it is easy to see, hence we omit the proof. 

\begin{lemma}\label{parallelogram-completion}

For $1 \le i < j \le n$ and $\omega \in S_{n}$, 

\begin{enumerate}

\item if $m_{i,a}(\omega) > m_{j,b}(\omega)$ with $a, b \ge 2$ then $m_{i,(a-1)}(\omega) > m_{j,(b-1)}(\omega)$, 
\item if $m_{i,a}(\omega) > m_{j,b}(\omega)$ with $a < c_{i}(\omega)$ and $ b < c_{j}(\omega)$ then $m_{i,(a+1)}(\omega) > m_{j,(b+1)}(\omega)$.

\end{enumerate}

\end{lemma}

Figure \ref{fig:parallelogram-completion} shows a visualization of Lemma 
{\rmfamily \ref{parallelogram-completion}}.

\begin{figure}[htbp]
\begin{center}
 \includegraphics[width=70mm]
{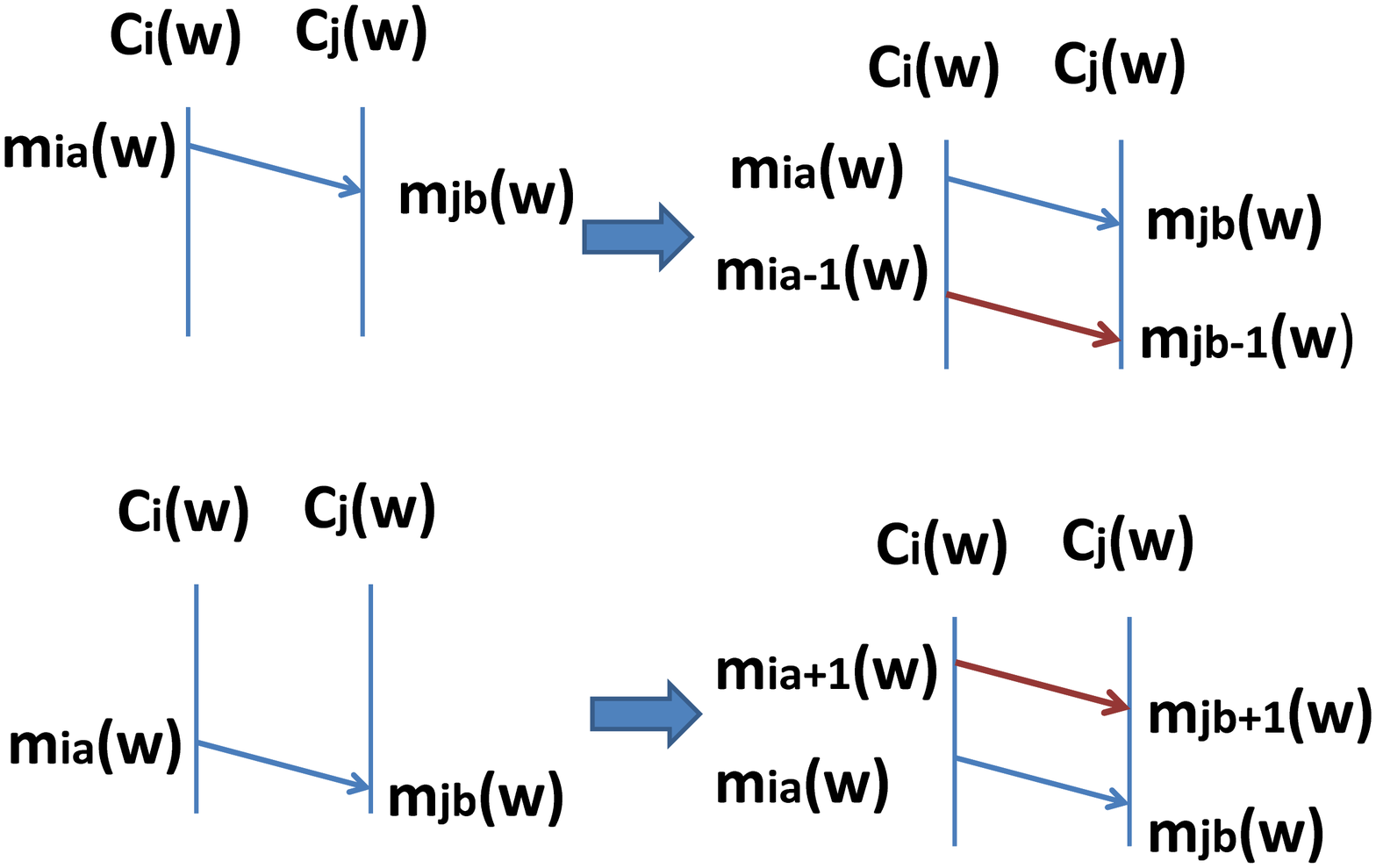}
 \caption{  }
 \label{fig:parallelogram-completion}
\end{center}
\end{figure}

\begin{lemma}

For $1 \le i < j \le n$ and $\omega \in S_{n}$ if 
$m_{i,p}(\omega) > m_{j,q}(\omega)$ for some $1 \le p \le c_{i}(\omega)$ and 
$1 \le q \le c_{j}( \omega)$ then 
$\{ k | j < k, \omega(i) > \omega(k) \} \subset 
\{ l | j < l, \omega(j) > \omega(l) \} $. 
\end{lemma}

\begin{proof}

By Lemma {\rmfamily \ref{denoncourt-231modoki}} we have $\omega(i) < \omega(j)$ because $m_{i,p}(\omega) > m_{j,q}(\omega)$ for some $1 \le p \le c_{i}(\omega)$ and 
$1 \le q \le c_{j}( \omega)$. 
 Hence we have $\{ k | j < k, \omega(i) > \omega(k) \} \subset 
\{ l | j < l, \omega(j) > \omega(l) \} $.

\end{proof}

\begin{lemma}\label{extension-of-cj-chain}

If $m_{i,p}(\omega) > m_{j,q}(\omega)$ for some $\omega \in S_{n}$, $p \in [c_{i}(\omega)]$ and $q \in [c_{j}(\omega)]$, then we have $c_{j}(\omega) \ge c_{i}(\omega) + q -p$.

\end{lemma}

\begin{proof}

Set $m_{i,p}(\omega) = (0, \cdots ,0, \displaystyle\overbrace{p}^{i}, \cdots , \displaystyle\overbrace{x}^{j}, \cdots ), 
m_{j,q}(\omega) = (0, \cdots ,0, \displaystyle\overbrace{0}^{i}, \cdots  ,\displaystyle\overbrace{q}^{j}, \cdots  )$. 
Then we have $x = p - c_{i,j}(\omega)$ and $c_{i,j}(\omega) \le p-q$  
 because $m_{i,p}(\omega) > m_{j,q}(\omega)$ and $x \le q$. 
Also we have $c_{i}(\omega) = c_{i,j}(\omega) + \sharp \{ k | j < k, \omega(i) > \omega(k) \} \le c_{i,j}(\omega) + c_{j}(\omega) \le p-q + c_{j}(\omega)$. 
Hence we have $c_{j}(\omega) \ge c_{i}(\omega) + q -p$.

\end{proof}

\begin{figure}[htbp]
\begin{center}
 \includegraphics[width=70mm]
{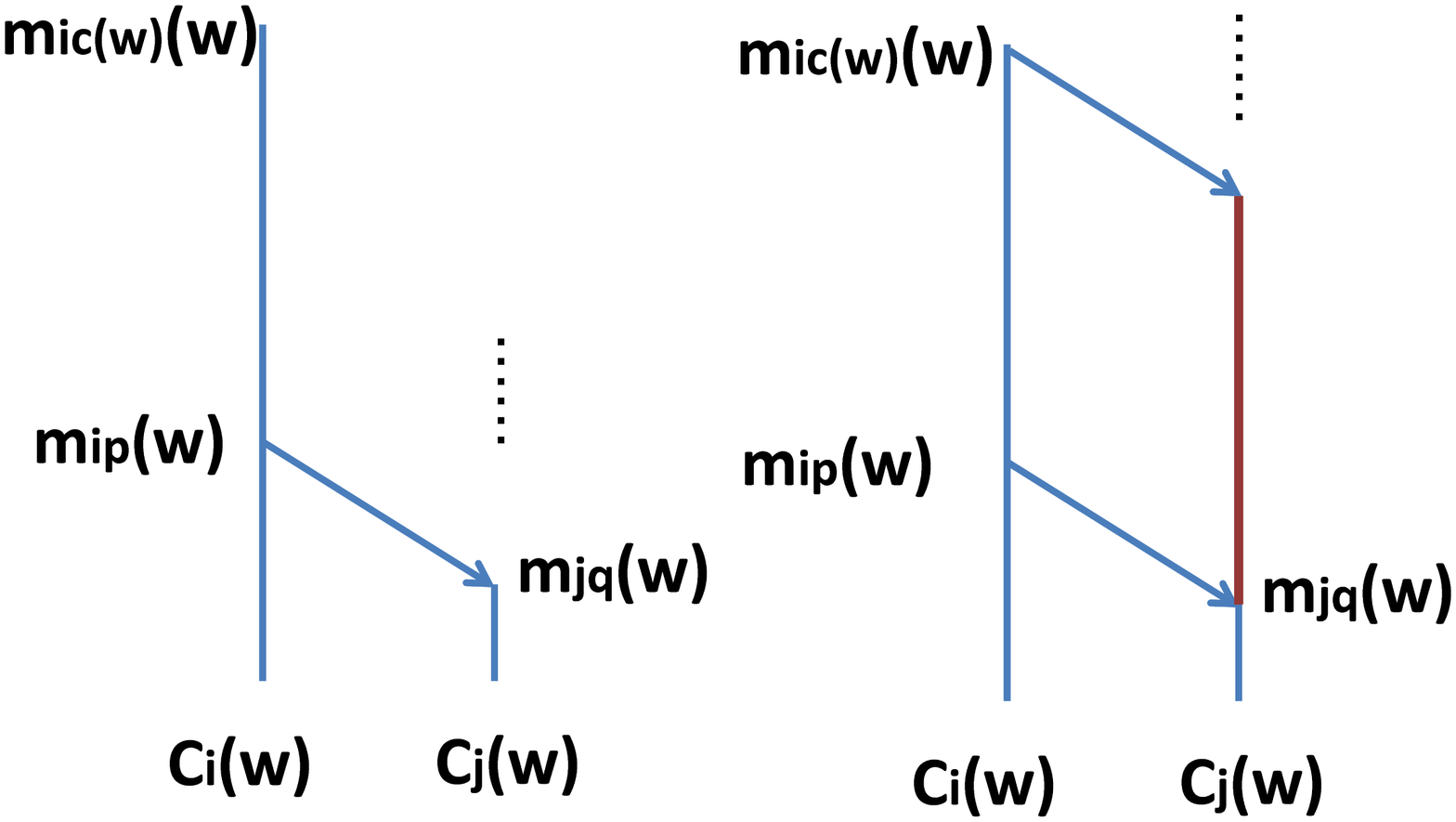}
 \caption{  }
 \label{fig:extension-of-cj-chain}
\end{center}
\end{figure}

The above  Lemma {\rmfamily \ref{extension-of-cj-chain}} says that if $m_{i,p}(\omega) > m_{j,q}(\omega)$ with $1 \le i < j \le n, p \in [c_{i}(\omega)]$ and $q \in [c_{j}(\omega) ]$  then
 there exists $c_{i}(\omega) + q -p \in [c_{j}(\omega)]$ such that 
$m_{i,c_{i}(\omega)}(\omega) >  m_{j,c_{i}(\omega) + q -p}(\omega)$
 by Lemma   {\rmfamily \ref{parallelogram-completion}}.
Figure \ref{fig:extension-of-cj-chain} shows a visualization of Lemma {\rmfamily \ref{extension-of-cj-chain}}.

For root poset $M_{\omega}$ we have the following observation.

\begin{lemma}\label{C4-parallelogram-goto-parallelogram}

For $\omega \in S_{n}$ if $M_{\omega}$ has a $C_{4}$-parallelogram-pattern then $M_{\omega}$ has a parallelogram-pattern.

\end{lemma}

\begin{proof}

By assumption there exists $\{ m_{i,a}(\omega), m_{i,b}(\omega), m_{j,c}(\omega), m_{j,d}(\omega) \} \subset M_{\omega}$  such that 
$ m_{i,a}(\omega) > m_{i,b}(\omega) >  m_{j,d}(\omega)  >  m_{j,c}(\omega)$ 
for some $1 \le i < j \le n, b < a \in [c_{i}(\omega)]$ 
 $ c < d \in [c_{j}(\omega)] $ with   $a + c = b + d$. 
It is easy to see that  $b \ge d$ because $m_{i,b}(\omega) >  m_{j,d}(\omega)$.

We will show that it is possible to construct a parallelogram-pattern poset from \\ $\{ m_{i,a}(\omega), m_{i,b}(\omega), m_{j,c}(\omega), m_{j,d}(\omega) \} $ by induction on $b-d$.  

If $b=d$ then $c_{i}(\omega) \ge a \ge b+1 > d$ and by Lemma {\rmfamily \ref{extension-of-cj-chain}} we have $d+1 \le c_{j}(\omega)$.

Consider the subposet $\{ m_{i,a}(\omega), m_{i,b}(\omega), m_{j,c+1}(\omega), m_{j,d+1}(\omega) \} \subset M_{\omega}$ where $i,j,a,b,c,d$ are as above.

We have $m_{i,a}(\omega) > m_{j,d+1}(\omega)$ 
because $m_{i,b}(\omega) > m_{j,d}(\omega)$ and hence $ m_{i,a}(\omega) \ge 
m_{i,b+1}(\omega) > m_{j,d+1}(\omega)$. 
 Also $m_{i,b}(\omega) > m_{j,c+1}(\omega)$ because $m_{i,b}(\omega) > m_{j,d}(\omega) \ge  m_{j,c+1}(\omega)$. 
Obviously $m_{i,b}(\omega)$ and $m_{j,d+1}(\omega)$ are incomparable. 
Therefore the subposet $\{ m_{i,a}(\omega), m_{i,b}(\omega), m_{j,c+1}(\omega), m_{j,d+1}(\omega) \} $ is parallelogram-pattern poset. 

Assume that we can construct  a parallelogram-pattern poset from  $\{ m_{i,a}(\omega), m_{i,b}(\omega), m_{j,c}(\omega), m_{j,d}(\omega) \} $ for $b-d \le k-1$. 

If $b-d = k$ then we consider a subposet $\{ m_{i,a}(\omega), m_{i,b}(\omega), m_{j,c+1}(\omega), m_{j,d+1}(\omega) \} \subset M_{\omega}$ where $i,j,a,b,c,d$ are as above.  
We have $m_{i,a}(\omega) > m_{j,d+1}(\omega)$ 
because $m_{i,b}(\omega) > m_{j,d}(\omega)$ and hence $ m_{i,a}(\omega) \ge 
m_{i,b+1}(\omega) > m_{j,d+1}(\omega)$. 
 Also $m_{i,b}(\omega) > m_{j,c+1}(\omega)$ because $m_{i,b}(\omega) > m_{j,d}(\omega) \ge  m_{j,c+1}(\omega)$. 
If $ m_{i,b}(\omega) > m_{j,d+1}(\omega)$ then we can construct a parallelogram-pattern poset by the assumption because $b-(d+1) = k-1$. 
If $m_{i,b}(\omega)$ and $m_{j,d+1}(\omega)$ are incomparable then the poset \\
$\{ m_{i,a}(\omega), m_{i,b}(\omega), m_{j,c+1}(\omega), m_{j,d+1}(\omega) \}$ 
is a parallelogram-pattern poset. This completes the proof. 

\end{proof}

\begin{lemma}\label{B2-equiv-parallelogram}

For $\omega \in S_{n}$ the poset $M_{\omega}$ has a $B_{2}$-pattern if and only if  it has a parallelogram-pattern. 

\end{lemma}

\begin{proof}

If $M_{\omega}$ has a parallelogram-pattern then obviously it has a $B_{2}$-pattern. Conversely we assume that $M_{\omega}$ has a $B_{2}$-pattern. 

Let $\{ m_{i,a}(\omega), m_{j,b}(\omega), m_{k,c}(\omega), m_{l,d}(\omega) \} $ with $i, j, k, l \in \mathbb{N}, a \in [c_{i}(\omega)], b \in [c_{j}(\omega)], 
c \in [c_{k}(\omega)]$ and $d \in [c_{l}(\omega)]$ 
be a $B_{2}$-pattern subposet of $M_{\omega}$ 
where $m_{i,a}(\omega)$ (resp. $m_{l,d}(\omega)$) is the maximum (resp. minimum) element and $m_{j,b}(\omega)$ and $m_{k,c}(\omega)$ are incomparable.
We can set $j < k$ without loss of generality and
 hence we have $i \le j < k \le l $.\\

{\bf Case.1} (The case of $d \ge 2$)

We have $m_{i,a-d-1}(\omega) > m_{l,1}(\omega)$ because $m_{i,a}(\omega) > m_{l,d}(\omega)$ and by Lemma 
{\rmfamily \ref{parallelogram-completion}}.
 Hence the poset $\{ m_{i,a}(\omega), m_{i,a-d+1}(\omega), m_{l,d}(\omega), 
m_{l,1}(\omega)      \}$ is either a parallelogram-pattern poset or a $C_{4}$-parallelogram-pattern poset. By Lemma {\rmfamily \ref{C4-parallelogram-goto-parallelogram}} the poset $M_{\omega}$ has a parallelogram-pattern for both cases. \\

{\bf Case.2} (The case of $d =1$ and $i=j$)

We have $c_{l}(\omega) \ge c_{i}(\omega) + 1 -b \ge a-b+1 $ because $m_{i,b}(\omega) > m_{l,1}(\omega)$ and by 
Lemma {\rmfamily \ref{extension-of-cj-chain}}. Also we have $a-b+1 \ge 2$ and $c_{l}(\omega) \ge 2$ because $i=j$ and $m_{i,a}(\omega) > m_{j,b}(\omega)$.

Hence the poset $\{ m_{i,b+1}(\omega), m_{i,b}(\omega), m_{l,2}(\omega), 
m_{l,1}(\omega)      \}$ is either a parallelogram-pattern poset or a $C_{4}$-parallelogram-pattern poset because $m_{i,b}(\omega) > m_{l,1}(\omega)$ and by 
Lemma {\rmfamily \ref{parallelogram-completion}}. 

By Lemma {\rmfamily \ref{C4-parallelogram-goto-parallelogram}} the poset $M_{\omega}$ has a parallelogram-pattern for both cases. \\

{\bf Case.3} (The case of $d =1$ and $k=l$)

In this case we have $m_{i,a}(\omega) > m_{k,c}(\omega)  \ge
 m_{k,d}(\omega)$ so we have $c > d \ge 1$. 

By Lemma {\rmfamily \ref{parallelogram-completion}} 
the poset $\{ m_{i,a}(\omega), m_{i,a-1}(\omega), m_{k,c}(\omega), m_{k,c-1}(\omega) \} $ is either a parallelogram-pattern poset or a $C_{4}$-parallelogram-pattern poset. 
By Lemma {\rmfamily \ref{C4-parallelogram-goto-parallelogram}} the poset $M_{\omega}$ has a parallelogram-pattern for both cases. \\

Next we will consider the case of $d=1$ with $i < j < k<l$.\\

{\bf Case.4} (The case of $b \ge 2$ or $c \ge 2$ with $d =1$ and $ i < j < k<l$)

We will consider the case of $b \ge 2$ and for the case of $ c \ge 2$ we can use the same argument. 
Because $m_{i,a}(\omega) > m_{j,b}(\omega)$ we have $a \ge 2$ and 
 $m_{i,a-1}(\omega) > m_{j,b-1}(\omega)$ by Lemma  {\rmfamily \ref{parallelogram-completion}} 

Hence the poset 
$\{  m_{i,a}(\omega),m_{i,a-1}(\omega),m_{j,b}(\omega),m_{j,b-1}(\omega)        \}$  is either a parallelogram-pattern poset or a $C_{4}$-parallelogram-pattern poset.
By Lemma {\rmfamily \ref{C4-parallelogram-goto-parallelogram}} the poset $M_{\omega}$ has a parallelogram-pattern for both cases. \\

{\bf Case.5} (The case of $b = c = d =1$ with $ i < j < k<l$ and $c_{j}(\omega) \ge 2$ or $c_{k}(\omega) \ge 2$)

We will consider the case of $c_{j}(\omega) \ge 2$ and for the case of $c_{k}(\omega)$ we can use the same argument.

From Lemma {\rmfamily \ref{extension-of-cj-chain}} we obtain $c_{l}(\omega) \ge c_{j}(\omega) + 1 -1 \ge 2$ because $m_{j,1}(\omega) > m_{l,1}(\omega)$ and 
$c_{j}(\omega) \ge 2$. By Lemma {\rmfamily \ref{parallelogram-completion}} we have  $m_{j,2}(\omega) > \exists m_{l,2}(\omega)$.
Hence the poset 
$\{  m_{j,2}(\omega), m_{j,1}(\omega),  m_{l,2}(\omega), m_{l,1}(\omega)   \}$ 
is  either a parallelogram-pattern poset or a $C_{4}$-parallelogram-pattern poset.
By Lemma {\rmfamily \ref{C4-parallelogram-goto-parallelogram}} the poset $M_{\omega}$ has a parallelogram-pattern for both cases. \\

{\bf Case.6} (The case of $b = c = d =1$ with $ i < j < k<l$ and $c_{j}(\omega)  = c_{k}(\omega) =1$)

We have $\omega(i) < \omega(j)$ and $\omega(k) < \omega(l)$ because $m_{i,a}(\omega) > m_{j,b}(\omega)$ and $m_{k,c}(\omega) > m_{l,d}(\omega)$.

Set $m_{i,a}(\omega) = 
(0, \cdots ,0, \displaystyle\overbrace{a}^{i}, \cdots , \displaystyle\overbrace{x}^{l}, \cdots )$
 and 
$m_{l,1}(\omega) = 
(0, \cdots ,0, \displaystyle\overbrace{0}^{i}, \cdots , 0, \displaystyle\overbrace{1}^{l}, \cdots )$. We obtain $x \ge 1$ so there exists  $p > l$ such that 
$\omega(i) > \omega(p)$.
It is easy to see that $\{ y | j < y, \omega(j) > \omega(y) \} = \{ p \}$ because $c_{j}(\omega) = 1$.
Hence  $\omega(j) < \omega(k)$ and then we get 
$\omega(i) < \omega(j) < \omega(k) < \omega(l)$.

Now we have 
$m_{j,1}(\omega) = 
(0, \cdots ,0, \displaystyle\overbrace{1}^{j},1,  \cdots , 1, \displaystyle\overbrace{1}^{k},1 \cdots ,1 , \displaystyle\overbrace{1}^{l},1 \cdots 
1, \displaystyle\overbrace{0}^{p}, 0 \dots 0 )$ 
and \\
$m_{k,1}(\omega) = 
(0, \cdots ,0, \displaystyle\overbrace{0}^{j},0,  \cdots , 0, \displaystyle\overbrace{1}^{k},1 \cdots ,1 , \displaystyle\overbrace{1}^{l},1, \cdots ,1,  \displaystyle\overbrace{0}^{p}, 0 \dots 0 )$ 
because $c_{k}(\omega) = 1$ and $\omega(k) > \omega(i) > \omega(p)$.  
Then we get $m_{j,1}(\omega)  > m_{k,1}(\omega) $ and this contradicts the assumption that $m_{j,1}(\omega)$ and $ m_{k,1}(\omega)$ are incomparable. 
Therefore the {\bf case.6} never happens. \\

This completes the proof. 
\end{proof}

\begin{lemma}\label{parallelogram-equiv-3412-3421}

For $\omega \in S_{n}$ the poset $M_{\omega}$ has a parallelogram-pattern if and only if $\omega$ has a $3412$-pattern or a $3421$-pattern.

\end{lemma}

\begin{proof}

Suppose that $\omega$ has a $3412$-pattern. Then there exists $i < j < k < l$
 such that \\ 
$st(\omega(i) \omega(j) \omega(k) \omega(l)) = 3412$ and we obtain 
$c_{i}(\omega) \ge 2$ and $c_{j}(\omega) \ge 2$. 
We have 

\begin{flushleft}

$m_{i, c_{i}(\omega)}(\omega) = (0, \cdots ,0, \displaystyle\overbrace{c_{i}(\omega)}^{i},\cdots , \displaystyle\overbrace{p}^{j}, \cdots ) $

$m_{i, c_{i}(\omega) -1}(\omega) = (0, \cdots ,0, \displaystyle\overbrace{c_{i}(\omega) -1 }^{i},\cdots , \displaystyle\overbrace{p-1}^{j}, \cdots ) $

$m_{j, 2}(\omega) = (0, \cdots ,0, \displaystyle\overbrace{0}^{i},\cdots ,0, \displaystyle\overbrace{2}^{j}, \cdots ) $

$m_{j, 1}(\omega) = (0, \cdots ,0, \displaystyle\overbrace{0}^{i},\cdots ,0, \displaystyle\overbrace{1}^{j}, \cdots ) $

\end{flushleft}

with $p \ge 2$ because $\omega(i) > \omega(k), \omega(l)$. \\

{\bf Claim} 

$m_{i,c_{i}(\omega) -1}(\omega) > m_{j, 1}(\omega)$\\

For $x \le j$ the $x$-th entry of $m_{j, 1}(\omega)$ is less than that of 
$m_{i, c_{i}(\omega) -1}(\omega)$.

For $x > j$ the $x$-th entry of $m_{j, 1}(\omega)$ is $0$ or $1$. 
If that of $m_{j, 1}(\omega)$ equals to  $1$ then  
$(j,y) \notin Inv(\omega)$ for $j < y \le x$
 and hence $\omega(i) < \omega(j) < \omega(x)$. 
Therefore $(i,y) \notin Inv(\omega)$ for $j < y \le x$  
and the $x$-th entry of 
$m_{i, c_{i}(\omega) -1}(\omega)$ is $(p-1) \ge 1$. 
The  $x$-th entry of of 
$m_{j, 1}(\omega)$ is less than that  of $m_{i, c_{i}(\omega) -1}(\omega)$ if 
that of $m_{j, 1}(\omega)$ is $0$.  
Hence we have $m_{i,c_{i}(\omega)-1}(\omega) > m_{j, 1}(\omega)$.

For the set 
$\{ m_{i, c_{i}(\omega)}(\omega), m_{i, c_{i}(\omega) -1}(\omega), m_{j, 2}(\omega), m_{j, 1}(\omega) \}$  we obtain $ m_{i, c_{i}(\omega)}(\omega) > m_{i, c_{i}(\omega) -1}(\omega) >  m_{j, 1}(\omega)$. 
Also we have $ m_{i, c_{i}(\omega)}(\omega) > m_{j,2}(\omega) >  m_{j, 1}(\omega)$ by Lemma  {\rmfamily \ref{parallelogram-completion}}.

Then the induced subposet is either a parallelogram-pattern poset or a $C_{4}$-parallelogram-pattern poset.
 By Lemma {\rmfamily \ref{C4-parallelogram-goto-parallelogram}} the poset $M_{\omega}$ has a parallelogram-pattern for both cases. 
We can use the same argument if $\omega$ has a $3421$-pattern. \\

Suppose that  $M_{\omega}$ has a parallelogram-pattern poset  
$\{ m_{i,a}(\omega), m_{i,b}(\omega), m_{j,c}(\omega), m_{j,d}(\omega) \}$ with $1 \le i < j \le n, b < a \in [c_{i}(\omega)]$ and 
$ d < c \in [c_{j}(\omega)] $ and $a + d = b + c$ where 
$m_{i,a}(\omega)$ (resp. $m_{j,d}(\omega)$) is the maximum (resp. minimum) element and $m_{i,b}(\omega)$ and $m_{j,c}(\omega)$ are incomparable. In particular $c \ge 2$ because $c > d \ge 1$.

From Lemma {\rmfamily \ref{denoncourt-231modoki}} we have $\omega(i) < \omega(j)$. Put

\begin{flushleft}

$m_{i, a}(\omega) = (0, \cdots ,0, \displaystyle\overbrace{a}^{i},\cdots , \displaystyle\overbrace{x}^{j}, \cdots ) $

$m_{j,c}(\omega) = (0, \cdots ,0, \displaystyle\overbrace{0}^{i},\cdots , \displaystyle\overbrace{c}^{j}, \cdots ) $

\end{flushleft}

where $x \ge c \ge 2$ hence $\sharp \{ y | j < y, \omega(i) > \omega(y) \} \ge 2$ and  there exists $j < y_{1} < y_{2}$ such that $\omega(i) > \omega(y_{1})$ and $\omega(i) > \omega(y_{2})$.
Then we have $st(\omega(i) \omega(j) \omega(y_{1}) \omega(y_{2})) = 3412 \ {\rm or} \ 3421$. 

This completes the proof.

\end{proof}

From Lemma {\rmfamily \ref{B2-equiv-parallelogram}} and Lemma {\rmfamily \ref{parallelogram-equiv-3412-3421}} we obtain the following result.

\begin{theorem}

$M_{\omega}$ is a $B_{2}$-free poset if and only if $\omega$ is a  $3412-3421$-avoiding permutation.

\end{theorem}

{\bf Acknowledgment}

The author wishes to thank Kento Nakada for his valuable advices.

\renewcommand{\refname}{REFERENCE}

\end{document}